\documentclass[a4paper, reqno]{amsart} 
\usepackage[margin=3.5cm]{geometry}

\usepackage{newlfont, color}   
\usepackage{amsthm, amssymb, amsmath, latexsym}
\usepackage{graphicx}
\usepackage{amsfonts, mathrsfs}
\usepackage{esint}
\usepackage{comment}

\theoremstyle{plain}                     
\begingroup
\newtheorem{teo}{Theorem}[section]

\endgroup
      
\theoremstyle{definition}                
\begingroup

\newtheorem{oss}[teo]{Remark}
\endgroup

\newcommand{\ep}{\varepsilon}
\newcommand{\Oe}{\Omega_\ep}

\newcommand{\R}{\mathbb R}

\newcommand{\C}{\mathbb C}
\newcommand{\M}{\mathbb M}
\newcommand{\ac}{r_\C}
\newcommand{\bc}{R_\C}
\newcommand{\rk}{r_K}
\newcommand{\Rk}{R_K}
\newcommand{\B}[1]{B(#1)}
\newcommand{\A}[0]{\mathbb{A}}

\newcommand{\dist}{\mathrm{dist}}

\newcommand{\sym}{\mathrm{sym}}

\newcommand{\md}{\mathbb M^{3\times 3}_D}

\newcommand{\mthree}{\mathbb{M}^{3\times 3}}

\newcommand{\deb}{\rightharpoonup}

\newcommand{\intom}[1]{\int_{\Omega}{#1\, dx}}
\newcommand{\intome}[1]{\int_{\Omega_{\ep}}{#1\, dx}}
\newcommand{\intomm}[1]{\int_{\omega}{#1\, dx'}}
\newcommand{\intt}[1]{\int_{-\frac{1}{2}}^{\frac{1}{2}}{#1\, dx_3}}

\newcommand{\pep}[0]{\phi^{\ep}}

\newcommand{\nep}[0]{\nabla_{\ep}}

\newcommand{\vep}[0]{\varphi^{\ep}}

\newcommand{\twh}[0]{\widetilde{W}_{hard}}

\newcommand{\tr}[1]{\text{tr }#1}

\newcommand{\m}[1]{(#1)'}
\newcommand{\pzero}{\bar p}
\newcommand{\puno}{\hat p}
\newcommand{\pdue}{p_\perp}

\numberwithin{equation}{section}

\title[Linearized plastic plate models derived via $\Gamma$-convergence]{Linearized plastic plate models as $\Gamma$-limits of 3D finite elastoplasticity} 
 
\author[E. Davoli]{Elisa Davoli} 

\address[E. Davoli]{Scuola Internazionale Superiore di Studi Avanzati, via Bonomea 265, 34136 Trieste (Italy); current address: Department of Mathematical Sciences, Carnegie-Mellon University, Pittsburgh, PA, USA}
\email{davoli@sissa.it; edavoli@andrew.cmu.edu}

 
\begin{document}
\begin{abstract}
The subject of this paper is the rigorous derivation of reduced models for a thin plate by means of $\Gamma$-convergence, in the framework of finite plasticity. Denoting by $\ep$ the thickness of the plate, we analyse the case where the scaling factor of the elasto-plastic energy is of order $\ep^{2\alpha-2}$, with $\alpha\geq 3$. According to the value of $\alpha$, partially or fully linearized models are deduced, which correspond, in the absence of plastic deformation, to the Von K\'arm\'an plate theory and the linearized plate theory.  
\end{abstract}
\maketitle
\section{Introduction}
The rigorous identification of lower dimensional models for thin structures is a classical question in mechanics. In the early 90's a rigorous approach to dimension reduction problems has emerged in the framework of nonlinear elasticity (\cite{A-B-P, L-R}). This approach is based on $\Gamma$-convergence: a variational convergence which guarantees, roughly speaking, convergence of minimizers of the three-dimensional energies to minimizers of the reduced models. In the seminal papers \cite{FJM, FJM2}, a hierarchy of limit models has been identified  by $\Gamma$-convergence methods for nonlinearly elastic thin plates. Different limit models have been deduced according to the scaling of the applied body forces in terms of the thickness parameter. In particular, high scalings of the applied forces lead at the limit to linearized models.

The purpose of this paper is to deduce some linearized reduced models for thin plates in the framework of finite plasticity. {{We remark that different schools in finite plasticity are still competing and a generally accepted model is still lacking (see e.g. \cite{Ber}). We shall adopt here a mathematical model introduced in \cite{CHM, Men, M}. }}We shall consider a three-dimensional plate of small thickness, whose elastic behaviour is nonlinear and whose plastic response is that of finite plasticity with hardening. We assume that the reference configuration of the plate is the set
$$\Omega_{\ep}:=\omega\times \big(-\tfrac{\ep}{2},\tfrac {\ep}{2}\big),$$
where $\omega$ is a domain in $\R^2$ and $\ep>0$ is the thickness parameter. {{Following the lines of  \cite{Lee} and \cite{Man},}} we consider deformations $\varphi\in W^{1,2}(\Oe;\R^3)$ that fulfill the multiplicative decomposition
$$\nabla \varphi(x)=F_{el}(x)F_{pl}(x)\quad\text{for a.e. }x\in\Oe,$$
where $F_{el}\in L^2(\Oe;\mthree)$ represents the elastic strain, $F_{pl}\in L^2(\Oe;SL(3))$ is the plastic strain and $SL(3):=\{F\in\mthree: \det F=1\}.$ { To guarantee coercivity in the plastic strain variable, we suppose to be in a hardening regime. More precisely, }the stored energy associated to a deformation $\varphi$ and to its elastic and plastic strains is expressed as follows:
\begin{eqnarray*}
\nonumber \cal{E}(\varphi,F_{pl})&:=&\intome{W_{el}(\nabla \varphi(x) F_{pl}^{-1}(x))}+\intome{W_{hard}(F_{pl}(x))}\\
&=&\intome{W_{el}(F_{el}(x))}+\intome{W_{hard}(F_{pl}(x))},
\end{eqnarray*}
where $W_{el}$ is a frame-indifferent elastic energy density and $W_{hard}$, which is finite only on a compact subset of $SL(3)$ having the identity as an interior point, describes hardening.
The plastic dissipation is expressed by means of a dissipation distance $D:\mthree\times\mthree\to [0,+\infty]$, which is given via a positively 1-homogeneous potential $H_D$, and represents the minimum amount of energy that is dissipated when the system moves from a plastic configuration to another one (see Section \ref{prel}).

We are interested in studying the asymptotic behaviour of sequences of pairs $(\vep, F_{pl}^{\ep})$ whose total energy per unit thickness satisfies
\begin{equation}
\label{eqintro}
\frac{1}{\ep}\Big(\cal{E}(\vep, F_{pl}^{\ep})+\ep^{\alpha-1}\intome{D(F^{\ep, 0}_{pl}, F^{\ep}_{pl})}\Big)\leq C\ep^{2\alpha-2},
\end{equation}
where $\alpha\geq 3$ is a positive parameter and $(F^{\ep, 0})\subset L^2(\Omega_{\ep};SL(3))$ is a given sequence representing preexistent plastic strains. It was proved in \cite{FJM2} that in the absence of plastic deformation $(F^{\ep,0}=F_{pl}=Id)$ these energy scalings lead to the Von K\'arm\'an plate theory for $\alpha=3$ and to the linear plate theory for $\alpha>3$. The scaling of the dissipation energy is motivated by its linear growth (see \eqref{prd2}). In analogy with the results of \cite{FJM2} in the framework of nonlinear elasticity, we expect these scalings to correspond to partially or fully linearized plastic models. 

On a portion of the lateral surface $$\gamma_d\times \big(-\tfrac{\ep}{2},\tfrac{\ep}{2}\big),$$
where $\gamma_d\subset \partial \omega$ {{has positive $\cal{H}^1$- measure}}, we prescribe a boundary datum 
$$\phi^{\ep}(x):=\Big(\begin{array}{c}x'\\x_3\end{array}\Big)+\Big(\begin{array}{c}\ep^{\alpha-1}u^0(x')\\\ep^{\alpha-2}v^0(x')\end{array}\Big)-\ep^{\alpha-2}x_3\nabla v^0(x')\\
$$
for $x=(x', \ep x_3)\in\Omega_{\ep}$, 
where $u^0\in W^{1,\infty}(\omega;\R^2)$ and $v^0\in W^{2,\infty}(\omega)$. This structure of the boundary conditions is compatible with that of the minimal energy configurations in the absence of plastic deformations (see Remark \ref{whyalpha}). 

Assuming $\ep\to 0$, we first show that, given any sequence of pairs $(\vep, F_{pl}^{\ep})$ satisfying \eqref{eqintro} and the boundary conditions
\begin{equation}
\label{bdcintro}
\vep=\pep\quad\cal{H}^2\text{ - a.e. on }\gamma_d\times \big(-\tfrac {\ep}{2},\tfrac {\ep}{2}\big),
\end{equation}
 the deformations $\vep$ converge to the identity deformation on the mid-section of the plate, and the plastic strains $F_{pl}^{\ep}$ tend to the identity matrix. More precisely, defining $\Omega:=\omega\times\big(-\tfrac12,\tfrac12\big)$ and $\psi^{\ep}(x):=(x',\ep x_3)$ for every $(x',x_3)\in\overline{\Omega}$, and assuming 
 $$F^{\ep, 0}_{pl}\circ \psi^{\ep}=Id+\ep^{\alpha-1}p^{\ep,0}$$
 with
 \begin{equation}
 \label{hpindat}
 p^{\ep,0}\deb p^0\quad\text{weakly in }L^2(\Omega;\mthree),
 \end{equation}
 we show that
$$y^{\ep}:=\vep\circ \psi^{\ep}\to\Big(\begin{array}{c}x'\\0\end{array}\Big)$$
strongly in $W^{1,2}(\Omega;\R^3)$ and
$$P^{\ep}:=F_{pl}^{\ep}\circ \psi^{\ep}\to Id$$
strongly in $L^2(\Omega;\mthree)$.
To express the limit functional, we introduce and study the compactness properties of some linearized quantities associated with the scaled deformations and plastic strains: the in-plane displacements
$$u^{\ep}(x'):=\frac{1}{\ep^{\alpha-1}}\intt{\Big(\Big(\begin{array}{c}y^{\ep}_1\\y^{\ep}_2\end{array}\Big)-x'\Big)}$$
for a.e. $x'\in\omega$, the out-of-plane displacements
$$v^{\ep}(x'):=\frac{1}{\ep^{\alpha-2}}\intt{y^{\ep}_3(x)},$$
for a.e. $x'\in\omega$, and the linearized plastic strains
$$p^{\ep}(x):=\frac{P^{\ep}(x)-Id}{\ep^{\alpha-1}}$$
for a.e. $x\in\Omega$. In Theorem \ref{liminfineq} we show that, under assumptions \eqref{eqintro}, \eqref{bdcintro} and \eqref{hpindat} 
\begin{eqnarray*}
&&u^{\ep}\deb u\quad\text{weakly in }W^{1,2}(\omega;\R^2),\\
&&v^{\ep}\to v\quad\text{strongly in }W^{1,2}(\omega),\\
&&p^{\ep}\deb p\quad\text{weakly in }L^2(\Omega;\mthree),
\end{eqnarray*}
for some $u\in W^{1,2}(\omega;\R^2)$, $v\in W^{2,2}(\omega)$ and $p\in L^2(\Omega;\mthree)$ such that $\tr{p}=0$, and
$$u=u^0,\quad v=v^0,\quad\nabla v=\nabla v^0\quad \cal{H}^1\text{-a.e. on }\gamma_d.$$
In Theorems \ref{liminfineq}, \ref{limsupineq} and \ref{minpts} we show that the $\Gamma$-limit functional is expressed in terms of the limit quantities $u,v$ and $p$, and is given by
\begin{eqnarray*}
\cal{J}_{\alpha}(u,v,p):=\intom{Q_2(\sym\nabla' u-x_3(\nabla')^2 v-p')}+\intom{\B{p}}+\intom{H(p-p^0)}
\end{eqnarray*}
for $\alpha>3$, and
\begin{eqnarray*}
\cal{J}_3(u,v,p)&:=&\intom{Q_2\big(\sym\nabla' u+\tfrac{1}{2}\nabla' v\otimes\nabla' v-x_3(\nabla')^2 v-p'\big)}+\intom{\B{p}}\\&+&\intom{H(p-p^0)}
\end{eqnarray*}
for $\alpha=3$.
In the previous formulas, $\nabla'$ denotes the gradient with respect to $x'$, $p'$ is the $2\times 2$ minor given by the first two rows and columns of the map $p$, and $Q_2$ and $B$ are positive definite quadratic forms on $\M^{2\times 2}$ and $\mthree$, respectively, for which an explicit characterization is provided (see Sections \ref{prel} and \ref{comp}). In the absence of plastic dissipation $(p^0=p=0)$ the two $\Gamma$-limits reduce to the functionals deduced in \cite{FJM2} in the context of nonlinear elasticity. 

We remark, anyway, that in constrast with the problem studied in \cite{FJM2}, the limit functionals $\cal{J}_{\alpha}$ and $\cal{J}_3$ cannot be, in general, expressed in terms of two-dimensional quantities only because the limit plastic strain $p$ depends nontrivially on the $x_3$ variable (see Section \ref{finsec}).\\

The setting of the problem and some proof arguments are very close to those of \cite{MS}, where it is shown that linearized plasticity can be obtained as $\Gamma$-limit of finite plasticity. 

The proof of the compactness and the liminf inequality relies on the rigidity estimate due to Friesecke, James and M\"uller (\cite[Theorem 3.1]{FJM}). This theorem can be applied owing to the presence of the hardening term, which provides one with a uniform bound on the $L^{\infty}$ norm of the scaled plastic strains $P^{\ep}$. The construction of the recovery sequence is obtained by combining the arguments in \cite[Sections 6.1 and 6.2]{FJM2} and \cite[Lemma 3.6]{MS}. For this construction we need to assume that { $\gamma_d$ is the finite union of disjoint (nontrivial) closed intervals (i.e., maximally connected sets) in  $\partial \omega$, }the convergence in \eqref{hpindat} is strong in $L^1(\Omega;\mthree)$ and the maps $p^{\ep,0}$ are uniformly bounded in $L^{\infty}(\Omega;\mthree)$.

In \cite{MS} the authors proved also the convergence of quasistatic evolutions in finite plasticity to a quasistatic evolution in linearized plasticity. The question whether an analogous convergence result can be established in the present context for thin plates is adressed in \cite{D2}.\\

The paper is organized as follows: in Section \ref{prel} we recall some preliminary results and we discuss the formulation of the problem. Section \ref{comp} is devoted to prove some compactness results and liminf inequalities, while in Section \ref{rec} we show that the lower bounds obtained in Section \ref{comp} are optimal. Finally, in Section \ref{finsec} we deduce convergence of almost minimizers of the three-dimensional energies to minimizers of the limit functionals and we discuss some examples.
\section{Preliminaries and setting of the problem}
\label{prel}
Let $\omega\subset \R^2$ be a connected, bounded open set with Lipschitz boundary. Let $\ep>0$. We assume the set $\Omega_{\ep}:=\omega\times\big(-\tfrac \ep2,\tfrac \ep2\big)$ to be the reference configuration of a finite-strain elastoplastic plate.

 We suppose that the boundary $\partial\omega$ is partitioned into the union of two disjoint sets $\gamma_d$ and $\gamma_n$ and their common boundary, where $\gamma_d$ is { such that $\cal{H}^1(\gamma_d)>0$}. We denote by $\Gamma_{\ep}$ the portion of the lateral surface of the plate given by $\Gamma_{\ep}:=\gamma_d\times\big(-\tfrac \ep2,\tfrac \ep2\big)$. On $\Gamma_{\ep}$ we prescribe a boundary datum of the form
\begin{equation}
\label{defbddat}
\phi^{\ep}(x):=\Big(\begin{array}{c}x'\\x_3\end{array}\Big)+\Big(\begin{array}{c}\ep^{\alpha-1}u^0(x')\\\ep^{\alpha-2}v^0(x')\end{array}\Big)-\ep^{\alpha-2}x_3\nabla v^0(x')
\end{equation}
for $x=(x',\ep x_3)\in\Omega_{\ep}$, 
where $u^0\in W^{1,\infty}(\omega;\R^2)$, $v^0\in W^{2,\infty}(\omega)$ and $\alpha\geq 3$. 

We assume that every deformation $\varphi\in W^{1,2}(\Oe;\R^3)$ fulfills the multiplicative decomposition
$$\nabla \varphi(x)=F_{el}(x)F_{pl}(x)\quad\text{for a.e. }x\in\Oe,$$
where $F_{el}\in L^2(\Oe;\mthree)$ represents the elastic strain, $F_{pl}\in L^2(\Oe;SL(3))$ is the plastic strain and $SL(3):=\{F\in\mthree: \det F=1\}.$ The stored energy associated to a deformation $\varphi$ and to its elastic and plastic strains can be expressed as follows:
\begin{eqnarray}
\nonumber \cal{E}(\varphi,F_{pl})&:=&\intome{W_{el}(\nabla \varphi(x) F_{pl}^{-1}(x))}+\intome{W_{hard}(F_{pl}(x))}\\
\label{nsenergy}&=&\intome{W_{el}(F_{el}(x))}+\intome{W_{hard}(F_{pl}(x))},
\end{eqnarray}
where $W_{el}$ is the elastic energy density and $W_{hard}$ describes hardening.\\
{\bf Properties of the elastic energy}\\
We assume that $W_{el}:\mthree\to [0,+\infty]$ satisfies
\begin{itemize}
\item[(H1)] $W_{el}\in C^1(\mthree_+),\quad W_{el}\equiv +\infty \text{ on }\mthree\setminus \mthree_+$,
\item [(H2)] $W_{el}(Id)=0,$
\item  [(H3)] $W_{el}(RF)=W_{el}(F)\quad\text{for every }R\in SO(3),\, F\in \mthree_+,$
\item [(H4)] $W_{el}(F)\geq c_1 \dist^2(F;SO(3))\quad\text{for every }F\in \mthree_+,$
\item [(H5)] $|DW_{el}(F)F^T|\leq c_2(W_{el}(F)+1)\quad\text{for every }F\in\mthree_+.$
\end{itemize}
Here $c_1,c_2$ are positive constants, $\mthree_+:=\{F\in\mthree:\, \det F>0\}$ and $SO(3):=\{F\in\mthree_+:\, F^T F=Id\}$. We also assume that there exists a symmetric, positive semi-definite tensor $\C:\mthree\to \mthree$ such that, setting
$$Q(F):=\frac{1}{2}\C F:F\quad\text{for every }F\in\mthree,$$
the quadratic form $Q$ encodes the local behaviour of $W_{el}$ around the identity, namely
\begin{equation}
\label{quadrwel}
 \forall \delta>0\,\, \exists c_{el}(\delta)>0 \text{ such that }\forall F\in B_{c_{el}(\delta)}(0)\text{ there holds }|W_{el}(Id+F)-Q(F)|\leq \delta |F|^2.
\end{equation}
We note that \eqref{quadrwel} implies in particular that
$$\C =D^2W_{el}(Id),\quad \C_{ijkl}=\frac{\partial^2 W}{\partial F_{ij}\partial F_{kl}}(Id)\text{ for every }i,j,k,l\in\{1,2,3\}.$$
As remarked in \cite[Section 2]{MS}, the frame-indifference condition (H3) yields that
$$\C_{ijkl}=\C_{jikl}=\C_{ijlk}\text{ for every }i,j,k,l\in\{1,2,3\}$$
and
$$\C F=\C\, (\sym\, F) \quad\text{for every }F\in\mthree.$$
Hence, the quadratic form $Q$ satisfies:
$$Q(F)=Q(\sym\,F)\quad\text{for every }F\in\mthree$$
and by (H4) it is positive definite on symmetric matrices. Therefore, there exist two constants $\ac$ and $\bc$ such that 
\begin{equation}
\label{growthcondQ}
\ac |F|^2\leq Q(F)\leq \bc |F|^2\quad\text{for every }F\in\mthree_{\sym},
\end{equation}
and
\begin{equation}
\nonumber
|\C F|\leq 2\bc |F| \quad\text{for every }F\in\mthree_{\sym}.
\end{equation}

\noindent{\bf Properties of the hardening functional}\\
We assume that the hardening map $W_{hard}:\mthree\to [0,+\infty]$ is of the form
\begin{equation}
\label{prh1}
W_{hard}(F):=\begin{cases}\twh(F)&\text{for every }F\in K,\\
+\infty&\text{otherwise}.
\end{cases}
\end{equation}
Here $K$ is a compact set in $SL(3)$ that contains the identity as a relative interior point, and the map $\twh:\mthree\to [0,+\infty)$ fulfills
\begin{eqnarray}
\nonumber &&\twh\text{ is locally Lipschitz continuous},\\
\label{prh3} && \twh(Id+F)\geq c_3 |F|^2\quad\text{for every }F\in\mthree,
\end{eqnarray}
where $c_3$ is a positive constant.
We also assume that there exists a positive semi-definite quadratic form $B$ such that
\begin{eqnarray}
\nonumber && \forall \delta>0\, \exists c_h(\delta)>0 \text{ such that }\forall F\in B_{c_h(\delta)}(0)\text{ there holds }|\twh(Id+F)-\B{F}|\leq \delta \B{F}.\\
 \label{prh4}
\end{eqnarray} 
In particular, by the hypotheses on $K$ there exists a constant $c_k$ such that
\begin{eqnarray}
\label{prk1}&& |F|+|F^{-1}|\leq c_k\quad\text{for every }F\in K,\\
\label{prk2}&& |F-Id|\geq \frac{1}{c_k}\quad\text{for every }F\in SL(3)\setminus K.
\end{eqnarray}
Combining \eqref{prh3} and \eqref{prh4} we deduce also
\begin{equation}
\label{grbelowh}
\frac{c_3}{2} |F|^2\leq \B{F}\quad\text{for every }F\in\mthree.
\end{equation}
{\bf Dissipation functional}\\
Denote by $\md$ the set of trace-free symmetric matrices, namely
$$\md:=\{F\in\mthree_{\sym}: \tr F=0\}.$$
Let $H_{D}:\md \to [0,+\infty)$ be a convex, positively one-homogeneous function such that
\begin{equation}
\label{growthh}
\rk |F|\leq H_{D}(F)\leq \Rk |F|\quad\text{for every }F\in\md.
\end{equation}
We define the dissipation potential $H:\mthree\to [0,+\infty]$ as
\begin{equation}
\nonumber
H(F):=\begin{cases}H_{D}(F)&\text{if }F\in \md,\\
+\infty &\text{otherwise.}\end{cases}
\end{equation}
For every $F\in\mthree$, we consider the quantity
\begin{equation}
\label{distid}
D(Id,F):=\inf \Big\{\int_0^1{H(\dot{c}(t)c^{-1}(t))\,dt}: c\in C^1([0,1];\mthree_+),\, c(0)=Id,\, c(1)=F \Big\}.
\end{equation}
Note that {by the Jacobi's formula for the derivative of the determinant of a differentiable matrix-valued map} if $D(Id, F)<+\infty$, then $F\in SL(3)$.

We define the dissipation distance as the map $D:\mthree\times \mthree\to [0,+\infty]$, given by 
$$D(F_1,F_2):=\begin{cases}D(Id, {F_2}F_1^{-1})& \text{if }F_1\in\mthree_{+}, F_2\in\mthree\\ +\infty& \text{if }F_1\notin \mthree_{+}, F_2\in\mthree. 
\end{cases}$$
We note that the map $D$ satisfies the triangle inequality
\begin{equation}
\label{triang}
D(F_1,F_2)\leq D(F_1,F_3)+D(F_3,F_2)
\end{equation}
for every $F_1,F_2, F_3\in\mthree$.
\begin{oss}
 We remark that there exists a positive constant $c_4$ such that
\begin{eqnarray}
\label{prd1} &&D(F_1,F_2)\leq c_4\quad\text{for every }F_1,F_2\in K,\\
\label{prd2} && D(Id,F)\leq c_4|F-Id|\quad\text{for every }F\in K.
\end{eqnarray}
Indeed, by the compactness of $K$ and the continuity of the map $D$ on $SL(3)\times SL(3)$ (see \cite{M}), there exists a constant $\tilde{c}_4$ such that 
\begin{equation}
\label{quasiprd1}
D(F_1,F_2)\leq \tilde{c}_4\quad\text{for every }F_1,F_2\in K.
\end{equation} 
By the previous estimate, \eqref{prd2} needs only to be proved in a neighbourhood of the identity. More precisely, let $\delta>0$ be such that $\log F$ is well defined for $F\in K$ and $|F-Id|<\delta$. If $F\in K$ is such that $|F-Id|\geq \delta$, by \eqref{quasiprd1} we deduce
$$D(Id,F)\leq \frac{\tilde{c}_4}{\delta}|F-Id|.$$
If $|F-Id|<\delta$, taking $c(t)=\exp({t\log F})$ in \eqref{distid}, inequality \eqref{growthh} yields
$$D(Id,F)\leq H_{D}(\log F)\leq \Rk |\log F|\leq C|F-Id|$$
for every $F\in K$. Collecting the previous estimates we deduce \eqref{prd1} and \eqref{prd2}.
\end{oss}
\noindent{\bf Change of variable and formulation of the problem}\\
As usual in dimension reduction problems we perform a change of variable to formulate the problem on a domain independent of $\ep$. We consider the set $\Omega:=\omega \times \big(-\tfrac 12, \tfrac 12\big)$ and the map $\psi^{\ep}:\overline{\Omega}\to \overline{\Omega}_{\ep}$ given by 
$$\psi^{\ep}(x):=(x',\ep x_3)\quad\text{for every }x\in\overline{\Omega}.$$
To every deformation $\varphi \in W^{1,2}(\Oe;\R^3)$ satisfying 
$$\varphi(x)=\phi^{\ep}(x)\quad \cal{H}^2\text{- a.e. on }\Gamma_{\ep}$$
 and to every plastic strain $F_{pl}\in L^2(\Oe;SL(3))$ we associate the scaled deformation $y:=\varphi\circ \psi^{\ep}$ and the scaled plastic strain $P:=F_{pl}\circ \psi^{\ep}$. Denoting by $\Gamma_d$ the set $\gamma_d\times \big(-\tfrac 12, \tfrac 12\big),$ the scaled deformation satisfies the boundary condition 
\begin{equation}
\label{bddatum}
y(x)=\phi^{\ep}(x',\ep x_3)\quad\cal{H}^2\text{- a.e. on }\Gamma_d. 
\end{equation}
Applying this change of variable to \eqref{nsenergy}, the energy functional is now given by
$$\cal{I}(y,P):=\frac{1}{\ep}\cal{E}(\varphi,F_{pl})=\intom{W_{el}(\nep y(x) P^{-1}(x))}+\intom{W_{hard}(P(x))},$$
where $\nep y(x):=\big(\partial_1 y(x)\big|\partial_2 y(x)\big|\frac{1}{\ep} \partial_3 y(x)\big)$ for a.e. $x\in\Omega$. 

Denote by $\cal{A}_{\ep}(\phi^{\ep})$ the class of pairs $(y^{\ep},P^{\ep})\in W^{1,2}(\Omega;\R^3)\times L^2(\Omega;SL(3))$ such that \eqref{bddatum} is satisfied. 
We associate to each pair $(y^{\ep},P^{\ep})\in \cal{A}_{\ep}(\phi^{\ep})$ the scaled energy given by 
\begin{eqnarray}
&&\cal{J}^{\ep}_{\alpha}(y^{\ep},P^{\ep}):=
\frac{1}{\ep^{2\alpha-2}}\cal{I}(y^{\ep},P^{\ep})+\frac{1}{\ep^{\alpha-1}}\intom{D({P}^{\ep,0},P^{\ep})},
\label{scaleden}
\end{eqnarray}
where $\alpha\geq 3$ is the same exponent as in \eqref{defbddat} and $P^{\ep,0}$ is a map in $L^2(\Omega; SL(3))$, which represents a preexistent plastic strain.
\begin{oss}
\label{whyalpha}
We are interested in studying the asymptotic behaviour of sequences of pairs $(y^{\ep},P^{\ep})\in \cal{A}_{\ep}(\phi^{\ep})$ such that the scaled total energies $\cal{J}^{\ep}_{\alpha}(y^{\ep},P^{\ep})$ are uniformly bounded. This, in particular, holds for sequences of (almost) minimizers of
\begin{equation}
\label{funct}
\cal{I}(y,P)-\intom{f^{\ep}\cdot y},
\end{equation}
whenever the applied forces $f^{\ep}$ are of order $\ep^{\alpha}$, with $\alpha\geq 3$. In fact by \cite[Theorem 2]{FJM2}, in the absence of plastic deformation ($P^{\ep}\equiv Id$), the elastic energy on (almost) minimizing sequences scales like $\ep^{2\alpha-2}$. In order to have interaction between the elastic and the plastic energy at the limit we are lead to rescale also the hardening functional by $\ep^{2\alpha-2}$. Finally, the scaling of the dissipation functional is motivated by its linear growth and by the estimate \eqref{prd2}.

Our choice of the boundary datum is again motivated by \cite[Theorem 2]{FJM2}. Indeed, as remarked in the introduction, the structure of $\pep$ is compatible with the structure of (almost) minimizers of \eqref{funct} in absence of plastic deformation, as $\ep\to 0^+$.
\end{oss}
\section{Compactness results and liminf inequality}
\label{comp}
In this section we study compactness properties of sequences of pairs in $\cal{A}_{\ep}(\phi^{\ep})$ satisfying the uniform energy estimate
\begin{equation}
\label{engest}
\cal{J}^{\ep}_{\alpha}(y^{\ep},P^{\ep})\leq C\quad\text{for every }\ep.
\end{equation}

To state the compactness results it is useful to introduce the following notation: given $\varphi:\Omega\to \R^3$, we denote by $\varphi':\Omega\to \R^2$ the map 
$$\varphi':=\Big(\begin{array}{c}\varphi_1\\\varphi_2\end{array}\Big)$$
and for every $\eta\in W^{1,2}(\Omega)$ we denote by $\nabla'\eta$ the vector $\Big(\begin{array}{c}\partial_1 \eta\\\partial_2\eta\end{array}\Big)$.
Analogously, given a matrix $M\in \mthree$, we use the notation $M'$ to represent the minor
$$M':=\Big(\begin{array}{cc}M_{11}&M_{12}\\M_{21}&M_{22}\end{array}\Big).$$

Given a sequence of deformations $(y^{\ep})\subset W^{1,2}(\Omega;\R^3)$, we consider some associated quantities: the in-plane displacements
\begin{equation}
\label{inplane}
u^{\ep}(x'):=\frac{1}{\ep^{\alpha-1}}\intt{\big(\m{{y}^{\ep}}(x',x_3)-x'\big)}\quad\text{for a.e. }x'\in \omega,
\end{equation}
the out-of-plane displacements
\begin{equation}
\label{outofplane}
v^{\ep}(x'):=\frac{1}{\ep^{\alpha-2}}\intt{y^{\ep}_3(x',x_3)}\quad\text{for a.e. }x'\in \omega,
\end{equation}
and the first order moments
\begin{equation}
\label{1stmoment}
{\xi}^{\ep}(x'):=\frac{1}{\ep^{\alpha-1}}\intt{x_3\Big({y}^{\ep}(x',x_3)-\Big(\begin{array}{c}x'\\\ep x_3\end{array}\Big)\Big)}\quad\text{for a.e. }x'\in\omega.
\end{equation}

A key tool to establish compactness of in-plane and out-of-plane displacements is the following rigidity estimate due to Friesecke, James and M\"uller  \cite[Theorem 3.1]{FJM}.
\begin{teo}
\label{fjm}
Let $U$ be a bounded Lipschitz domain in $\mathbb{R}^n$, $n\geq 2$. Then there exists a constant $C(U)$ with the following property: for every $v\in W^{1,2}(U;\mathbb{R}^n)$ there is an associated rotation $R\in SO(n)$ such that
$$\| \nabla v-R\|_{L^2(U)}\leq C(U)\|\dist(\nabla v, SO(n))\|_{L^2(U)}.$$
\end{teo}
\begin{oss}
\label{ofjm}
The constant $C(U)$ in Theorem \ref{fjm} is invariant by translations and dilations of $U$ and is uniform for families of sets which are uniform bi-Lipschitz images of a cube.
\end{oss}
The rigidity estimate provided in Theorem \ref{fjm} allows us to approximate sequences of deformations whose distance of the gradient from $SO(3)$ is uniformly bounded, by means of rotations. More precisely, the following theorem holds true.
\begin{teo}
\label{compactbd1}
Assume that $\alpha\geq 3$. Let $(y^{\ep})$ be a sequence of deformations in $W^{1,2}(\Omega;\R^3)$ satisfying \eqref{bddatum} and such that
\begin{equation}
\label{elasticbd}
\|\dist(\nep y^{\ep}, SO(3))\|_{L^2(\Omega; \mthree)}\leq C\ep^{\alpha-1}.
\end{equation}
Then, there exists a sequence $(R^{\ep})\subset W^{1,\infty}(\omega; \mthree)$ such that for every $\ep>0$
\begin{eqnarray}
&&\label{rt1} R^{\ep}(x')\in SO(3)\quad\text{for every }x'\in \omega,\\
&&\label{rt2} \|\nep y^{\ep}-R^{\ep}\|_{L^2(\Omega;\mthree)}\leq C\ep^{\alpha-1},\\
&&\label{rt3} \|\partial_i R^{\ep}\|_{L^2(\omega;\mthree)}\leq C\ep^{\alpha-2},\,i=1,2\\
&&\label{rt4} \|R^{\ep}-Id\|_{L^2(\omega;\mthree)}\leq C\ep^{\alpha-2}.
\end{eqnarray} 
\end{teo}
\begin{proof}
Arguing as in \cite[Theorem 6 and Remark 5]{FJM2} we can construct a sequence of maps $R^{\ep}\in W^{1,\infty}(\omega;\mthree)$ satisfying \eqref{rt1}--\eqref{rt3}. To complete the proof of the theorem it remains only to prove \eqref{rt4}.

To this aim, we preliminarily recall that there exists a neighbourhood $U$ of $SO(3)$ where the projection $\Pi:U\to SO(3)$ onto $SO(3)$ is well defined. By Poincar\'e inequality, \eqref{rt3} yields 
\begin{equation}
\label{averrot}
\Big\|R^{\ep}-\fint_{\omega}{R^{\ep}\,dx'}\Big\|_{L^2(\omega;\mthree)}\leq C \ep^{\alpha-2}.
\end{equation}
On the other hand, by \eqref{rt1} we have
$$\dist^2\Big(\fint_{\omega}{R^{\ep}\,dx'},SO(3)\Big)\cal{L}^2(\omega)\leq \Big\|R^{\ep}-\fint_{\omega}{R^{\ep}\,dx'}\Big\|^2_{L^2(\omega;\mthree)}.$$
 Hence, by \eqref{averrot} for $\ep$ small enough we can define $\hat{R}^{\ep}:=\Pi (\fint_{\omega}{R^{\ep}\,dx'}),$ which fulfills
$$\Big|\hat{R}^{\ep}-\fint_{\omega}{R^{\ep}\,dx'}\Big|\leq C\Big\|R^{\ep}-\fint_{\omega}{R^{\ep}\,dx'}\Big\|_{L^2(\omega;\mthree)}\leq C\ep^{\alpha-2}.$$

\begin{equation}
\nonumber
\|\hat{R}^{\ep}-R^{\ep}\|_{L^2(\omega;\mthree)}\leq\Big\|\hat{R}^{\ep}-\fint_{\omega}{R^{\ep}\,dx'}\Big\|_{L^2(\omega;\mthree)}+\Big\|\fint_{\omega}{R^{\ep}\,dx'}-R^{\ep}\Big\|_{L^2(\omega;\mthree)} \leq C\ep^{\alpha-2}.
\end{equation}
To prove \eqref{rt4} it is now enough to show that
\begin{equation}
\label{claimrot}
|\hat{R}^{\ep}-Id|\leq C\ep^{\alpha-2}.
\end{equation}
To this purpose, we argue as in \cite[Section 4.2, Lemma 13]{LM2}. We consider the sequences
\begin{eqnarray*}
\label{rtep}&&\tilde{R}^{\ep}:=(\hat{R}^{\ep})^T R^{\ep},\\
\label{ytep}&&\tilde{y}^{\ep}:=(\hat{R}^{\ep})^T y^{\ep}-c^{\ep},\\
\label{utep}&&\tilde{u}^{\ep}(x'):=\frac{1}{\ep^{\alpha-1}}\intt{\big(\m{\tilde{y}^{\ep}}(x',x_3)-x'\big)}\quad\text{for a.e. }x'\in\omega,\\
\label{vtep}&&\tilde{v}^{\ep}(x'):=\frac{1}{\ep^{\alpha-2}}\intt{\tilde{y}^{\ep}_3(x',x_3)}\quad\text{for a.e. }x'\in\omega,\\
\label{xitep}&&\tilde{\xi}^{\ep}(x'):=\frac{1}{\ep^{\alpha-1}}\intt{x_3\Big(\tilde{y}^{\ep}(x',x_3)-\Big(\begin{array}{c}x'\\\ep x_3\end{array}\Big)\Big)}\quad\text{for a.e. }x'\in\omega,
\end{eqnarray*}
where the constants $c^{\ep}$ are chosen in such a way that $$\intom{\big(\tilde{y}^{\ep}(x)-x\big)}=0.$$
By \cite[Lemma 1 and Corollary 1]{FJM2}, there exist $\tilde{u}\in W^{1,2}(\omega;\R^2)$, $\tilde{v}\in W^{2,2}(\omega)$ and $\tilde{\xi}\in W^{1,2}(\omega;\R^3)$ such that
\begin{eqnarray}
\label{fcvutep}&& \tilde{u}^{\ep}\deb \tilde{u}\quad\text{weakly in }W^{1,2}(\omega;\R^2),\\
\label{fcvvtep}&& \tilde{v}^{\ep}\to \tilde{v}\quad\text{strongly in }W^{1,2}(\omega),\\
\label{fcvxitep}&& \tilde{\xi}^{\ep}\deb \tilde{\xi}\quad\text{weakly in }W^{1,2}(\omega;\R^3).
\end{eqnarray}
We now write $u^{\ep}, v^{\ep}$ and $\xi^{\ep}$ in terms of $\tilde{u}^{\ep}, \tilde{v}^{\ep}$ and $\tilde{\xi}^{\ep}$. We have
\begin{equation}
\label{notiltil}
\Big(\begin{array}{c}\ep^{\alpha-1}u^{\ep}(x')\\\ep^{\alpha-2}v^{\ep}(x')\end{array}\Big)=(\hat{R}^{\ep}-Id)\Big(\begin{array}{c}x'\\0\end{array}\Big)+\hat{R}^{\ep}\Big(\begin{array}{c}\ep^{\alpha-1}\tilde{u}^{\ep}(x')\\\ep^{\alpha-2}\tilde{v}^{\ep}(x')\end{array}\Big)+\hat{R}^{\ep}c^{\ep},
\end{equation}
for a.e. $x'\in\omega$ and
\begin{equation}
\label{1stmomrel}
\xi^{\ep}(x')=\frac{1}{12\ep^{\alpha-2}}(\hat{R}^{\ep}-Id)e_3+\hat{R}^{\ep}\tilde{\xi}^{\ep}(x')\quad\text{for a.e. }x'\in\omega.
\end{equation}
By \eqref{fcvxitep} there exists a constant $C$ such that $\|\tilde{\xi}^{\ep}\|_{L^2(\gamma_d;\R^3)}\leq C$ for every $\ep$. Moreover, by \eqref{defbddat} and \eqref{bddatum} there holds
$$\xi^{\ep}(x')=\frac{1}{\ep^{\alpha-1}}\intt{x_3\Big(\pep(x',\ep x_3)-\Big(\begin{array}{c}x'\\\ep x_3\end{array}\Big)\Big)}=\Big(\begin{array}{c}-\frac{1}{12}\nabla' v^0(x')\\0\end{array}\Big)\quad\cal{H}^1\text{- a.e. on }\gamma_d,$$
hence $(\xi^{\ep})$ is uniformly bounded in $L^2(\gamma_d;\R^3)$. Therefore, by \eqref{1stmomrel} we deduce 
\begin{equation}
\label{3colrotminid}
|(\hat{R}^{\ep}-Id)e_3|\leq C \ep^{\alpha-2}\|\xi^{\ep}-\hat{R}^{\ep}\tilde{\xi}^{\ep}\|_{L^2(\Gamma_d;\R^3)}\leq C\ep^{\alpha-2},
\end{equation}
for every $\ep$. Since $\hat{R}^{\ep}\in SO(3)$, \eqref{3colrotminid} implies that
\begin{equation}
\label{3rowrotminid}
|(\hat{R}^{\ep}-Id)^Te_3|\leq C\ep^{\alpha-2}
\end{equation}
for every $\ep$ and there exists a sequence $(\hat{Q}^{\ep})\subset SO(2)$ such that
\begin{equation}
\label{minorrotminid}
|\m{\hat{R}^{\ep}}-\hat{Q}^{\ep}|\leq C\ep^{\alpha-2}.
\end{equation}
Now, without loss of generality we can assume that
\begin{equation}
\label{normalization}
\int_{\gamma_d}{x'\,d\cal{H}^1(x')}=0\quad\text{and }\int_{\gamma_d}{|x'|^2\,d\cal{H}^1(x')}=c>0.
\end{equation}
By \eqref{fcvutep} and \eqref{fcvvtep} we have $\|\tilde{u}^{\ep}\|_{L^2(\gamma_d;\R^2)}+\|\tilde{v}^{\ep}\|_{L^2(\gamma_d)}\leq C$ for every $\ep$. On the other hand \eqref{defbddat} and \eqref{bddatum} imply that
$$u^{\ep}(x')=u^0(x')\quad\text{and}\quad v^{\ep}(x')=v^0(x')\quad\cal{H}^1\text{- a.e. on } \gamma_d,$$ 
hence both $(u^{\ep})$ and $(v^{\ep})$ are uniformly bounded in $L^2(\gamma_d;\R^2)$ and $L^2(\gamma_d)$, respectively. Therefore, by \eqref{notiltil} and \eqref{minorrotminid} we deduce 
\begin{equation}
\label{prelestrbar}
|(\hat{Q}^{\ep}-Id)x'+\m{\hat{R}^{\ep}c^{\ep}}|\leq C\ep^{\alpha-2}.
\end{equation}
 The two terms in the left hand side of \eqref{prelestrbar} are orthogonal in the sense of $L^2(\gamma_d;\R^2)$ by \eqref{normalization}, hence
\eqref{prelestrbar} implies that
$$\|(\hat{Q}^{\ep}-Id)x'\|^2_{L^2(\gamma_d;\R^2)}\leq C\ep^{2(\alpha-2)}.$$
Since $\hat{Q}^{\ep}\in SO(2)$, it satisfies
$$2|(\hat{Q}^{\ep}-Id)x'|^2=|\hat{Q}^{\ep}-Id|^2|x'|^2\quad\text{for every }x'\in\gamma_d.$$
Therefore, applying again \eqref{normalization} we obtain
\begin{equation}
\label{qhatest}
c|\hat{Q}^{\ep}-Id|^2=2\int_{\gamma_d}{|\hat{Q}^{\ep}-Id|^2 |x'|^2\,d\cal{H}^1(x')}\leq C\ep^{2(\alpha-2)}.
\end{equation}
Claim \eqref{claimrot} follows now by collecting \eqref{3colrotminid}--\eqref{minorrotminid} and \eqref{qhatest}.
\end{proof}

In the remaining of this section we shall establish some compactness results for the displacements defined in \eqref{inplane} and \eqref{outofplane}, and we shall prove a liminf inequality both for the energy functional and the dissipation potential. We first introduce the limit functional.

Let $\A:\mathbb{M}^{2\times 2}\to \mthree_{\sym}$ be the operator given by
$$\A (F):=\Bigg(\begin{array}{cc}\sym\,F&\hspace{-0.5 cm}\begin{array}{c}\lambda_1(F)\vspace{-0.1 cm}\\\lambda_2(F)\end{array}\vspace{-0.1 cm}\\\begin{array}{cc}\lambda_1(F)&\lambda_2(F)\end{array}&\hspace{-0.5 cm}\lambda_3(F)\end{array}\Bigg)\quad\text{for every }F\in\M^{2\times 2},$$
 where for every $F\in\M^{2\times 2}$ the triple $(\lambda_1(F),\lambda_2(F),\lambda_3(F))$ is the unique solution to the minimum problem
$$\min_{\lambda_i\in\R}Q\Bigg(\begin{array}{cc}\sym\,F&\hspace{-0.5 cm}\begin{array}{c}\lambda_1\vspace{-0.1 cm}\\\lambda_2\end{array}\vspace{-0.1 cm}\\\begin{array}{cc}\lambda_1&\lambda_2\end{array}&\hspace{-0.5 cm}\lambda_3\end{array}\Bigg).$$
We remark that for every $F\in\M^{2\times 2}$, $\A(F)$ is given by the unique solution to the linear equation
\begin{equation}
 \label{linearmin}
\C \A (F):\Bigg(\begin{array}{ccc}0&0&\lambda_1\\0&0&\lambda_2\\\lambda_1&\lambda_2&\lambda_3\end{array}\Bigg)=0\quad\text{for every }\lambda_1,\lambda_2,\lambda_3\in\R.
\end{equation}
This implies, in particular, that $\A$ is linear.

We define the quadratic form $Q_2:\mathbb{M}^{2\times 2}\to [0,+\infty)$ as
$$Q_2(F)=Q(\A (F))\quad\text{for every }F\in \mathbb{M}^{2\times 2}.$$
By properties of $Q$, we have that $Q_2$ is positive definite on symmetric matrices. We also define the tensor $\C_2:\mathbb{M}^{2\times 2}\to \mthree_{\sym}$, given by

\begin{equation}
\label{defc2}
\C_2 F:=\C\A (F)\quad\text{for every }F\in \mathbb{M}^{2\times 2}.
\end{equation}
We remark that by \eqref{linearmin} there holds
\begin{equation}
\label{nothird}
\C_2 F:G=\C_2 F:\Big(\begin{array}{cc}\sym\,G&0\\0&0\end{array}\Big)\quad\text{for every }F\in \mathbb{M}^{2\times 2},\,G\in\mthree
\end{equation}
and
$$Q_2(F)=\frac{1}{2}\C_2 F:\Big(\begin{array}{cc}\sym\,F&0\\0&0\end{array}\Big)\quad\text{for every }F\in \mathbb{M}^{2\times 2}.$$

Denoting by
$\cal{A}(u^0,v^0)$ the set of triples $(u,v,p)\in W^{1,2}(\Omega;\R^2)\times W^{2,2}(\Omega)\times L^2(\Omega;\md)$ such that 
$$u(x')=u^0(x'),\quad v(x')=v^0(x'),\quad\text{and } \nabla v(x')=\nabla v^0(x')\quad\cal{H}^1\text{ - a.e. on }\gamma_d,$$
we introduce the functionals $\cal{J}_{\alpha}:\cal{A}(u^0,v^0)\to [0,+\infty)$, given by
\begin{eqnarray}
\label{numlimfa}\cal{J}_{\alpha}(u,v,p):=\intom{Q_2(\sym\nabla' u-x_3(\nabla')^2 v-p')}+\intom{\B{p}}+\intom{H_D(p-p^0)}
\end{eqnarray}
for $\alpha>3$, and
\begin{eqnarray}
\nonumber\cal{J}_3(u,v,p)&:=&\intom{Q_2\big(\sym\nabla' u+\tfrac{1}{2}\nabla' v\otimes\nabla' v-x_3(\nabla')^2 v-p'\big)}+\intom{\B{p}}\\
\label{numlimf3}&+&\intom{H_D(p-p^0)},
\end{eqnarray}
for every $(u,v,p)\in \cal{A}(u^0,v^0)$. In the expressions of the functionals, $p^0$ is a given map in $L^2(\Omega;\md)$ that represents the history of the plastic deformations. 

Finally, for every sequence $(y^{\ep})$ in $W^{1,2}(\Omega;\R^3)$ satisfying both \eqref{bddatum} and \eqref{elasticbd}, we introduce the strains
\begin{equation}
\label{defgep}
G^{\ep}(x):=\frac{(R^{\ep}(x))^T\nep y^{\ep}(x)-Id}{\ep^{\alpha-1}}\quad\text{for a.e. }x\in\Omega,
\end{equation}
where the maps $R^{\ep}$ are the pointwise rotations provided by Theorem \ref{compactbd1}.

We are now in a position to state the main result of this section.
\begin{teo}
\label{liminfineq}
Assume that $\alpha\geq 3$. Let $(y^{\ep},P^{\ep})$ be a sequence of pairs in $\cal{A}_{\ep}(\phi^{\ep})$ satisfying 
\begin{equation}
\label{engest2}
\cal{I}(y^{\ep},P^{\ep})\leq C{\ep^{2\alpha-2}}
\end{equation}
for every $\ep>0$.
 Let $u^{\ep}$, $v^{\ep}$ and $G^{\ep}$ be defined as in \eqref{inplane}, \eqref{outofplane} and \eqref{defgep}, respectively. Then, there exists $(u,v,p)\in \cal{A}(u^0,v^0)$ such that, up to subsequences, there hold
\begin{eqnarray}
&&\label{cptyep}y^{\ep}\to \Big(\begin{array}{c}x'\\0\end{array}\Big)\quad\text{strongly in }W^{1,2}(\Omega;\R^3),\\
&&\label{cptuep} u^{\ep}\deb u\quad\text{weakly in }W^{1,2}(\omega;\R^2),\\
&&\label{cptvep} v^{\ep}\to v\quad\text{strongly in }W^{1,2}(\omega),\\
&&\label{cptnep3}\frac{\nabla' y^{\ep}_3}{\ep^{\alpha-2}}\to \nabla' v\quad\text{strongly in }L^2(\Omega;\R^2),
\end{eqnarray}
and the following estimate holds true
\begin{equation}
\label{3comphest}
\big\|\frac{y^{\ep}_3}{\ep}-x_3-\ep^{\alpha-3}v^{\ep}\big\|_{L^2(\Omega)}\leq C\ep^{\alpha-2}.
\end{equation}
Moreover, there exists $G\in L^2(\Omega;\mthree)$ such that
\begin{equation}
\label{cptGep}
G^{\ep}\deb G\quad\text{weakly in }L^2(\Omega;\mthree),
\end{equation}
and the $2\times 2$ submatrix $G'$ satisfies
\begin{equation}
\label{gaff}
G'(x', x_3) = G_0(x') - x_3 (\nabla')^2 v(x')\quad\text{for a.e. }x\in\Omega,
\end{equation}
where
\begin{eqnarray}
&&\label{Ga3} \sym\, G_0 = \frac{(\nabla' u+(\nabla' u)^T +\nabla' v\otimes \nabla' v)}{2}\quad\text{if } \alpha=3,\\
&&\label{Ga>3} \sym\, G_0 = \sym \nabla' u\quad\text{if }\alpha> 3.
\end{eqnarray}
The sequence of plastic strains $(P^{\ep})$ fulfills 
\begin{equation}
\label{Pep1} P^{\ep}(x)\in K\quad\text{for a.e. }x\in\Omega,
\end{equation}
and
\begin{equation}
\label{Pep2} \|P^{\ep}-Id\|_{L^2(\Omega;\mthree)}\leq C\ep^{\alpha-1}
\end{equation}
for every $\ep$. Moreover, setting 
\begin{equation}
\label{defpep}
p^{\ep}:=\frac{P^{\ep}-Id}{\ep^{\alpha-1}},
\end{equation} up to subsequences
\begin{equation}
\label{wconvpep}
p^{\ep}\deb p\quad\text{weakly in }L^2(\Omega;\mthree).
\end{equation}
Finally,
\begin{eqnarray}
\label{liminftot}
\intom{Q_2(\sym\, G'-p')}+\intom{B(p)}
\leq \liminf_{\ep\to 0} \frac{1}{\ep^{2\alpha-2}}\cal{I}(y^{\ep},P^{\ep}).\end{eqnarray}
If in addition 
\begin{equation}
\label{engest3}
\frac{1}{\ep^{\alpha-1}}\intom{D({P}^{\ep,0},{P}^{\ep})}\leq C\quad\text{for every }\ep>0
\end{equation}
and there exist a map $p^0\in L^2(\Omega;\mthree_D)$ and a sequence $({p}^{\ep,0})\subset L^2(\Omega;\mthree)$ such that $P^{\ep,0}=Id+\ep^{\alpha-1}p^{\ep,0}$, with ${p}^{\ep,0}\deb p^0$ weakly in $L^2(\Omega;\mthree)$, then
\begin{equation}
\label{liminfdiss}
\intom{H_{D}({p}-p^0)}\leq \liminf_{\ep\to 0} \frac{1}{\ep^{\alpha-1}}\intom{D({P}^{\ep,0},{P}^{\ep})}.
\end{equation}
\end{teo}
\begin{proof}
We first remark that by \eqref{engest2} there holds
\begin{equation}
\label{hardeningbd}
\intom{W_{hard}(P^{\ep})}\leq C\ep^{2\alpha-2},
\end{equation}
which, together with \eqref{prh1}, implies \eqref{Pep1}. On the other hand, combining \eqref{prh3} and \eqref{hardeningbd} we deduce
$$c_3\|P^{\ep}-Id\|^2_{L^2(\Omega;\mthree)}\leq \intom{\tilde{W}_{hard}(P^{\ep})}\leq C\ep^{2\alpha-2},$$
which in turn yields \eqref{Pep2} and \eqref{wconvpep}.

Let $R\in SO(3)$. By \eqref{prk1}, \eqref{Pep1} and \eqref{defpep} there holds
\begin{eqnarray*}
&&|\nep y^{\ep}-R|^2=|\nep y^{\ep}-RP^{\ep}+\ep^{\alpha-1}Rp^{\ep}|^2\leq 2\big(|\nep y^{\ep}(P^{\ep})^{-1}-R|^2|P^{\ep}|^2+\ep^{2\alpha-2}|p^{\ep}|^2\big)\\
&&\leq 2\,c_K^2|\nep y^{\ep}(P^{\ep})^{-1}-R|^2+2\ep^{2\alpha-2}|p^{\ep}|^2.
\end{eqnarray*}
Hence, the growth condition (H4) implies
$$\|\dist(\nep y^{\ep},SO(3))\|^2_{L^2(\Omega;\mthree)}\leq C\Big(\intom{W_{el}(\nep y^{\ep}(P^{\ep})^{-1})}+\ep^{2\alpha-2}\|p^{\ep}\|^2_{L^2(\Omega;\mthree)}\Big),$$
which in turn yields
\begin{equation}
\nonumber
\|\dist(\nep y^{\ep},SO(3))\|^2_{L^2(\Omega;\mthree)}\leq C\ep^{2\alpha-2}
\end{equation}
by \eqref{engest2} and \eqref{wconvpep}.

Due to \eqref{bddatum}, the deformations $(y^{\ep})$ fulfill the hypotheses of Theorem \ref{compactbd1}. Hence, we can construct a sequence $(R^{\ep})$ in $W^{1,\infty}(\omega;\mthree)$ satisfying \eqref{rt1}--\eqref{rt4}. 
Properties \eqref{cptyep}--\eqref{cptnep3} and \eqref{cptGep}--\eqref{Ga>3} follow arguing as in \cite[Lemma 1, Corollary 1 and Lemma 2]{FJM2}. The only difference is due to the fact that compactness is now achieved by using the boundary condition \eqref{bddatum}, instead of performing a normalization of the deformations $y^{\ep}$. Moreover the limit in-plane and out-of-plane displacements satisfy $u=u^0$, $v=v^0$ and $\nabla'v=\nabla' v^0$ $\cal{H}^1$- a.e. on $\gamma_d$.

By Poincar\'e inequality and the definition of $v^{\ep}$, there holds
$$\Big\|\frac{y^{\ep}_3}{\ep}-x_3-\ep^{\alpha-3}v^{\ep}\Big\|_{L^2(\Omega)}\leq C\Big\|\frac{\partial_3 y^{\ep}_3}{\ep}-1\Big\|_{L^2(\Omega)},$$
hence \eqref{3comphest} is a consequence of \eqref{rt2} and \eqref{rt4}.

Inequality \eqref{liminfdiss} follows by adapting \cite[Lemmas 3.4 and 3.5]{MS}. 

The proof of \eqref{liminftot} is based on an adaptation of \cite[Proof of Lemma 3.3]{MS}: we give a sketch for convenience of the reader. 
Fix $\delta>0$, let $O_{\ep}$ be the set
$$O_{\ep}:=\{x:\,\ep^{\alpha-1}|p^{\ep}(x)|\leq c_h(\delta)\}$$
and let $\chi_{\ep}$ be its characteristic function.
By \eqref{wconvpep} and by Chebyshev's inequality there holds
$$\cal{L}^3(\Omega\setminus O_{\ep})\leq C\ep^{2\alpha-2},$$
hence by \eqref{prh4} and \eqref{engest2}, we deduce
\begin{equation}
\label{liminfhard}
 \liminf_{\ep \to 0} \frac{1}{\ep^{2\alpha-2}}\intom{W_{hard}(P^{\ep})}\geq \liminf_{\ep \to 0} (1-\delta)\intom{B(p^{\ep})\chi_{\ep} }\geq(1-\delta)\intom{\B{p}}.
\end{equation}
To prove the liminf inequality for the elastic energy, we introduce the auxiliary tensors
\begin{equation}
\label{defwep}
w^{\ep}:=\frac{(P^{\ep})^{-1}-Id+\ep^{\alpha-1}p^{\ep}}{\ep^{\alpha-1}}=\ep^{\alpha-1}(P^{\ep})^{-1}(p^{\ep})^2.
\end{equation}
 By  \eqref{prk1} and \eqref{Pep1}, there exists a constant $C$ such that 
 \begin{equation}
 \label{linftypep}
 \ep^{\alpha-1}\|p^{\ep}\|_{L^{\infty}(\Omega;\mthree)}\leq C
 \end{equation} and 
 \begin{equation}
 \label{linftyweph}
 \ep^{\alpha-1}\|w^{\ep}\|_{L^{\infty}(\Omega;\mthree)}\leq C 
 \end{equation}  
 for every $\ep$. Furthermore, by \eqref{wconvpep},
 $$\|w^{\ep}\|_{L^1(\Omega;\mthree)}\leq C\ep^{\alpha-1}\quad\text{for every }\ep.$$
 By the two previous estimates it follows that $(w^{\ep})$ is uniformly bounded in $L^2(\Omega;\mthree)$ and  
 \begin{equation}
 \label{l2weph}
 w^{\ep}\deb 0\quad\text{weakly in }L^2(\Omega;\mthree).
 \end{equation} 
For every $\ep$ we consider the map
$$F^{\ep}:=\frac{1}{\ep^{\alpha-1}}\big((Id+\ep^{\alpha-1}G^{\ep})(P^{\ep})^{-1}-Id\big).$$
By the frame-indifference hypothesis (H3) there holds
$$W_{el}(\nep y^{\ep}(P^{\ep})^{-1})=W_{el}(Id+\ep^{\alpha-1}F^{\ep}).$$
On the other hand,
$$F^{\ep}=G^{\ep}+w^{\ep}-p^{\ep}+\ep^{\alpha-1}G^{\ep}(w^{\ep}-p^{\ep}).$$
Combining \eqref{cptGep}, \eqref{wconvpep} and \eqref{linftypep}--\eqref{l2weph} we deduce
\begin{equation}
\nonumber
F^{\ep}\deb G-p\quad\text{weakly in }L^2(\Omega;\mthree).
\end{equation}
Therefore, by \eqref{quadrwel} and arguing as in the proof of \eqref{liminfhard} we conclude that
\begin{eqnarray}
\nonumber
\intom{Q_2(\sym\, G'-p')}\leq \intom{Q(\sym\, G-p)}\\
\label{liminfelen}
\leq \liminf_{\ep\to 0} \frac{1}{\ep^{2\alpha-2}}\intom{W_{el}(\nep y^{\ep}(P^{\ep})^{-1})}.
\end{eqnarray}
Collecting \eqref{liminfhard} and \eqref{liminfelen}, we obtain \eqref{liminftot}.
\end{proof}
\section{Construction of the recovery sequence}
\label{rec}
In this section, under some additional hypotheses on the sequence $(p^{\ep,0})$ {{and on $\gamma_d$}}, we prove that the lower bound obtained in Theorem \ref{liminfineq} is optimal by exhibiting a recovery sequence. 
\begin{teo}
\label{limsupineq}
Assume that $\alpha\geq 3$ {{and $\gamma_d$ is a finite union of disjoint (nontrivial) closed intervals (i.e., maximally connected sets) in $\partial \omega$}}. Let $p^0\in L^{\infty}(\Omega;\md)$ be such that there exists a sequence $(p^{\ep,0})\subset L^{\infty}(\Omega;\md)$ satisfying
\begin{eqnarray}
&&\label{hpip1}\|{p}^{\ep,0}\|_{L^{\infty}(\Omega;\md)}\leq C\quad\text{for every }\ep,\\
&&\label{hpip2}{p}^{\ep,0}\to {p}^0\quad\text{strongly in }L^1(\Omega;\md).
\end{eqnarray}
Assume also that for every $\ep$ the map ${P}^{\ep,0}:=Id+\ep^{\alpha-1}{p}^{\ep,0}$ satisfies $\det {P}^{\ep,0}=1$.
Let $(u,v,p)\in\cal{A}(u^0,v^0)$. Then, there exists a sequence $(y^{\ep}, P^{\ep})\in\cal{A}_{\ep}(\pep)$ such that, defining $u^{\ep}, v^{\ep}$ and $p^{\ep}$ as in \eqref{inplane}, \eqref{outofplane} and \eqref{defpep}, we have
\begin{eqnarray}
&&\label{rscid} y^{\ep}\to \Big(\begin{array}{c}x'\\0\end{array}\Big)\quad\text{strongly in }W^{1,2}(\Omega;\R^3),\\
&&\label{rscu} u^{\ep}\to u\quad\text{strongly in }W^{1,2}(\omega;\R^2),\\
&&\label{rscv} v^{\ep} \to v\quad\text{strongly in }W^{1,2}(\omega),\\
&&\label{rscp} p^{\ep}\to p\quad\text{strongly in }L^2(\Omega;\mthree).
\end{eqnarray}
Moreover, 
\begin{eqnarray}
\label{limsuptot}\lim_{\ep\to 0}\cal{J}^{\ep}_{\alpha}(y^{\ep},P^{\ep})=\cal{J}_{\alpha}(u,v,p),
\end{eqnarray}
where $\cal{J}^{\ep}_{\alpha}$ and $\cal{J}_{\alpha}$ are the functionals introduced in \eqref{scaleden}, \eqref{numlimfa} and \eqref{numlimf3}.
\end{teo}
\begin{proof}
For the sake of simplicity we divide the proof into two steps.\\
\emph{Step 1}\\
Let $(u,v,p)\in \cal{A}(u^0,v^0)$. We first remark that by a standard approximation argument we may assume that $p\in C^{\infty}_c(\Omega;\md)$. Moreover, we claim that we can always reduce to the case where $u\in W^{1,\infty}(\omega;\R^2)$ and $v\in W^{2,\infty}(\omega)$. That is, we can approximate the pair $(u,v)$ in the sense of \eqref{rscu}--\eqref{rscv} by a sequence of pairs $(u^{\lambda},v^{\lambda})$ in $W^{1,\infty}(\omega;\R^2)\times W^{2,\infty}(\omega)$ satisfying the same boundary conditions as $(u,v)$ on $\gamma_d$, and such that, for $\alpha>3$,
\begin{eqnarray}
\nonumber&&\lim_{\lambda\to +\infty}\intom{Q_2\Big(\sym \nabla'u^{\lambda}-x_3(\nabla')^2 v^{\lambda}-p'\Big)}\\
\label{recvel>}&&=\intom{Q_2\Big(\sym \nabla'u-x_3(\nabla')^2 v-p'\Big)},
\end{eqnarray}
 whereas for $\alpha=3$ 
 \begin{eqnarray}
\nonumber
&&\lim_{\lambda\to +\infty}\intom{Q_2\Big(\sym \nabla'u^{\lambda}+\frac{1}{2}\nabla'v^{\lambda}\otimes \nabla'v^{\lambda}-x_3(\nabla')^2 v^{\lambda}-p'\Big)}\\
\label{reccvel}&&=\intom{Q_2\Big(\sym \nabla'u+\frac{1}{2}\nabla'v\otimes \nabla'v-x_3(\nabla')^2 v-p'\Big)}.
\end{eqnarray}
By { the hypotheses on $\gamma_d$, we may apply} \cite[Proposition A.2]{FJM}, { and} for every $\lambda>0$ we construct a pair $(u^{\lambda},v^{\lambda})\in W^{1,\infty}({\omega};\R^2)\times W^{2,\infty}(\omega)$, such that
$(u^{\lambda},v^{\lambda},p)\in \cal{A}(u^{0},v^{0})$, 
\begin{equation}
\label{lusinest}
\|u^{\lambda}\|_{W^{1,\infty}(\omega;\R^2)}+\|v^{\lambda}\|_{W^{2,\infty}(\omega)}\leq C\lambda,
\end{equation}
and setting
$$\omega^{\lambda}:=\{x'\in\omega: u^{\lambda}(x')\neq u(x')\text{ or }v^{\lambda}(x')\neq v(x')\},$$
there holds
\begin{equation}
\label{measset}
\lim_{\lambda\to +\infty}\lambda^2\cal{L}^2(\omega^{\lambda})=0. 
\end{equation}
Now, by \eqref{lusinest} we obtain
\begin{eqnarray*}
\|u^{\lambda}-u\|_{W^{1,2}(\omega;\R^2)}\leq C\big(\|u^{\lambda}-u\|_{L^{2}(\omega^{\lambda};\R^2)}+\|\nabla'u^{\lambda}-\nabla'u\|_{L^{2}(\omega^{\lambda};\M^{2\times 2})}\big)\\
\leq C\big(\|u\|_{L^{2}(\omega^{\lambda};\R^2)}+\|\nabla'u\|_{L^{2}(\omega^{\lambda};\M^{2\times 2})}+\lambda \big(\cal{L}^2(\omega^{\lambda})\big)^{\frac{1}{2}}\big)
\end{eqnarray*}
and, analogously
\begin{eqnarray*}
\|v^{\lambda}-v\|_{W^{2,2}(\omega;\R^2)}\leq C\big(\|v\|_{L^{2}(\omega^{\lambda})}+\|\nabla'v\|_{L^{2}(\omega^{\lambda};\R^2)}+\|(\nabla')^2 v\|_{L^{2}(\omega^{\lambda};\M^{2\times 2})}+\lambda \big(\cal{L}^2(\omega^{\lambda})\big)^{\frac{1}{2}}\big).
\end{eqnarray*}
Hence, by \eqref{measset} we deduce
\begin{equation}
\label{culam}
u^{\lambda}\to u\quad\text{strongly in }W^{1,2}(\omega;\R^2)
\end{equation}
and
\begin{equation}
\label{cvlam}
v^{\lambda}\to v\quad\text{strongly in }W^{2,2}(\omega),
\end{equation}
as $\lambda\to +\infty$. Therefore, in particular
\begin{equation}
\label{morest}
\nabla' v^{\lambda}\to \nabla'v\quad\text{strongly in }L^p(\omega;\R^2)\text{ for every }p\in[2,+\infty).
\end{equation}
By \eqref{culam}, \eqref{cvlam} and \eqref{morest} we obtain \eqref{recvel>} and \eqref{reccvel}.\\
\emph{Step 2}\\
To complete the proof of the theorem we shall prove that for every triple $(u,v,p)\in\cal{A}(u^0,v^0)$, with $u\in W^{1,\infty}({\omega};\R^2), v\in W^{2,\infty}(\omega)$ and $p\in C^{\infty}_c(\Omega;\md)$ we can construct a sequence $(y^{\ep}, P^{\ep})\in\cal{A}(\pep)$ satisfying \eqref{rscid}--\eqref{limsuptot}.

To this purpose, consider the functions
$$P^{\ep}:=\exp(\ep^{\alpha-1}p)\quad\text{and}\quad p^{\ep}:=\frac{1}{\ep^{\alpha-1}}(\exp(\ep^{\alpha-1}p)-Id).$$
Since $p\in C^{\infty}_c(\Omega;\md)$, it is immediate to see that $\det P^{\ep}(x)=1$ for every $\ep$ and for all $x\in \Omega$. Moreover, there exists $\ep_0>0$ such that 
$$P^{\ep}(x)\in K\quad\text{for every }x\in\Omega\text{ and for all }0\leq\ep<\ep_0,$$
and there holds 
$$p^{\ep}\to p\quad\text{uniformly in }\Omega,$$
which in turn implies \eqref{rscp}. Furthermore, $$\|P^{\ep}-Id\|_{L^{\infty}(\Omega;\mthree)}\leq C\ep^{\alpha-1},$$
and by \eqref{prh4}, for every $\delta>0$ there exists $\ep_{\delta}$ such that if $0\leq \ep<\ep_{\delta}$ there holds
\begin{eqnarray*}
\Big|\frac{1}{\ep^{2\alpha-2}}\intom{W_{hard}(P^{\ep})}-\intom{B(p^{\ep})}\Big|\leq \delta \intom{B(p^{\ep})}.
\end{eqnarray*}
By \eqref{rscp} we deduce that
\begin{equation}
\label{limhard}
\lim_{\ep\to 0}\frac{1}{\ep^{2\alpha-2}}\intom{W_{hard}(P^{\ep})}=\intom{B(p)}.
\end{equation}

To study the dissipation potential, we first remark that by \eqref{hpip1}, for $\ep$ small enough, there holds 
\begin{equation}
\label{datink}
\exp(\ep^{\alpha-1}p^{\ep,0}(x))({P}^{\ep,0})^{-1}(x)\in K\quad\text{for every }x\in\Omega.
\end{equation}
Hence, by \eqref{triang} and \eqref{prd2} the following estimate holds true:
\begin{eqnarray*}
\frac{1}{\ep^{\alpha-1}}\intom{D({P}^{\ep,0},P^{\ep})}&\leq& \frac{1}{\ep^{\alpha-1}}\intom{D({P}^{\ep,0},\exp(\ep^{\alpha-1}{p}^{\ep,0}))}\\
&+&\frac{1}{\ep^{\alpha-1}}\intom{D(\exp(\ep^{\alpha-1}{p}^{\ep,0}),\exp(\ep^{\alpha-1}p))}\\
\\&\leq&\frac{C}{\ep^{\alpha-1}}\intom{|\exp(\ep^{\alpha-1}{p}^{\ep,0})({P}^{\ep,0})^{-1}-Id|}\\
&+&\frac{1}{\ep^{\alpha-1}}\intom{D(Id, \exp(\ep^{\alpha-1}(p-{p}^{\ep,0})))}.\\
\end{eqnarray*}
By the positive homogeneity of $H_{D}$ and taking $c(t)=\exp(\ep^{\alpha-1}(p-{p}^{\ep,0})t)$ in \eqref{distid}, we obtain
$$\frac{1}{\ep^{\alpha-1}}\intom{D(Id, \exp(\ep^{\alpha-1}(p-{p}^{\ep,0})))}\leq\intom{H_D(p-{p}^{\ep,0})}.$$
On the other hand, by \eqref{hpip1} there holds
\begin{eqnarray*}
\intom{|\exp(\ep^{\alpha-1}{p}^{\ep,0})({P}^{\ep,0})^{-1}-Id|}\leq c_K\intom{|\exp(\ep^{\alpha-1}{p}^{\ep,0})-Id-\ep^{\alpha-1}{p}^{\ep,0}|}\leq C\ep^{2\alpha-2}.
\end{eqnarray*}
Collecting the previous estimates we deduce
$$\frac{1}{\ep^{\alpha-1}}\intom{D({P}^{\ep,0},P^{\ep})}\leq\intom{H_D(p-{p}^{\ep,0})}+C\ep^{\alpha-1},$$
which in turn, by \eqref{hpip2}, yields
\begin{equation}
\label{limdis}
\limsup_{\ep\to 0}\frac{1}{\ep^{\alpha-1}}\intom{D({P}^{\ep,0},P^{\ep})}\leq \intom{H_D(p-{p}^0)}.
\end{equation}

Let $d\in C^{\infty}_c(\Omega;\R^3)$ and consider the deformations
$$y^{\ep}(x):=\Big(\begin{array}{c}x'\\\ep x_3\end{array}\Big)+\ep^{\alpha-1}\Big(\begin{array}{c}u(x')-x_3\nabla' v(x')\\0\end{array}\Big)+\ep^{\alpha-2}\Big(\begin{array}{c}0\\v(x')\end{array}\Big)+\ep^{\alpha}\int_{-\frac{1}{2}}^{x_3}{d(x',s)\,ds}$$
for every $x\in\Omega$. It is immediate to see that the sequence $(y^{\ep})$ fulfills both \eqref{bddatum} and \eqref{rscid}. We note that
$$u^{\ep}(x')=u(x')+\ep\intt{\int_{-\frac{1}{2}}^{x_3}{d'(x',s)\,ds}}$$
and
$$v^{\ep}(x')=v(x')+\ep^2\intt{\int_{-\frac{1}{2}}^{x_3}{d_3(x',s)\,ds}}$$
for every $x'\in\omega$, hence both \eqref{rscu} and \eqref{rscv} hold true. To complete the proof of the theorem, it remains to show that for $\alpha>3$ 
\begin{equation}
\label{clrs1}\lim_{\ep\to 0}\frac{1}{\ep^{2\alpha-2}}\intom{W_{el}(\nep y^{\ep}(P^{\ep})^{-1})}
=\intom{Q\Big(\sym\Big(\begin{array}{c}\nabla'u-x_3(\nabla')^2 v\\0\end{array}\Big|d\Big)-p\Big)},
\end{equation}
and for $\alpha=3$,
\begin{eqnarray}
\nonumber&&\lim_{\ep\to 0}\frac{1}{\ep^{2\alpha-2}}\intom{W_{el}(\nep y^{\ep}(P^{\ep})^{-1})}\\
\nonumber&&=\intom{Q\Big(\sym\Big(\begin{array}{c}\nabla'u+\frac{1}{2}\nabla'v\otimes\nabla'v-x_3(\nabla')^2 v\\0\end{array}\Big|\begin{array}{c}d'\\d_3+|\nabla'v|^2\end{array}\Big)-p\Big)}.\\
\label{clrs2}
\end{eqnarray}
Indeed, if \eqref{clrs1} holds, then by a standard approximation argument we may assume that
$$Q\Big(\sym\Big(\begin{array}{c}\nabla'u-x_3(\nabla')^2 v\\0\end{array}\Big|d\Big)-p\Big)=Q_2\big(\sym\,\nabla'u-x_3(\nabla')^2 v-p'\big).$$
Analogously, if \eqref{clrs2} holds we may assume that
\begin{eqnarray*}
&&Q\Big(\sym\Big(\begin{array}{c}\nabla'u+\frac{1}{2}\nabla'v\otimes\nabla'v-x_3(\nabla')^2 v\\0\end{array}\Big|\begin{array}{c}d'\\d_3+\tfrac{|\nabla'v|^2}{2}\end{array}\Big)-p\Big)\\
&&=Q_2\Big(\sym\,\nabla'u+\frac{1}{2}\nabla'v\otimes\nabla'v-x_3(\nabla')^2 v-p'\Big).
\end{eqnarray*}
In both cases by \eqref{limhard}, \eqref{limdis}, and Theorem \ref{liminfineq}, we obtain \eqref{limsuptot}.

To prove \eqref{clrs1} and \eqref{clrs2} we first note that
\begin{eqnarray*}
\nep y^{\ep}=Id+\ep^{\alpha-1}\Big(\begin{array}{c}\nabla' u-x_3(\nabla')^2 v\\0\end{array}\Big|d\Big)+\ep^{\alpha-2}\Big(\begin{array}{cc}0&-\nabla'v\\(\nabla'v)^T&0\end{array}\Big)+O(\ep^{\alpha}).
\end{eqnarray*}
Hence, in particular, $\det (\nep y^{\ep})>0$ for $\ep$ small enough.
On the other hand, by the frame-indifference hypothesis (H3), there holds
$$W_{el}(\nep y^{\ep}(P^{\ep})^{-1})=W_{el}\Big(\sqrt{(\nep y^{\ep})^T\nep y^{\ep}}(P^{\ep})^{-1}\Big)\quad\text{ a.e. in }\Omega.$$
A direct computation yields
\begin{eqnarray*}
&&\sqrt{(\nep y^{\ep})^T\nep y^{\ep}}=Id+\ep^{\alpha-1}\sym\Big(\begin{array}{c}\nabla' u-x_3(\nabla')^2 v\\0\end{array}\Big|d\Big)\\
&&+\frac{\ep^{2\alpha-4}}{2}\Big(\begin{array}{cc}\nabla'v\otimes \nabla'v&0\\0&|\nabla' v|^2\end{array}\Big)+o(\ep^{\alpha-1}),
\end{eqnarray*}
and
$$W_{el}\big(\nep y^{\ep}(P^{\ep})^{-1}\big)=W_{el}\big(Id+\ep^{\alpha-1}M_{\alpha}+o(\ep^{\alpha-1})\big)\quad\text{ a.e. in }\Omega,$$
where
$$M_{\alpha}:=\begin{cases}\sym\Big(\begin{array}{c}\nabla' u-x_3(\nabla')^2 v\\0\end{array}\Big|d\Big)-p&\text{if }\alpha>3,\\
\sym\Big(\begin{array}{c}\nabla' u+\frac{1}{2}\nabla'v\otimes \nabla'v-x_3(\nabla')^2 v\\0\end{array}\Big|\begin{array}{c}d'\\d_3+\tfrac{|\nabla'v|^2}{2}\end{array}\Big)-p&\text{if }\alpha=3.\end{cases}$$
Fix $\delta>0$. For every $\alpha\geq 3$ we have $M_{\alpha}\in L^{\infty}(\Omega;\mthree)$, therefore for $\ep$ small enough 
$$\|\ep^{\alpha-1}M_{\alpha}+o(\ep^{\alpha-1})\|_{L^{\infty}(\Omega;\mthree)}\leq c_{el}(\delta).$$
 By \eqref{quadrwel}, we deduce
$$ \limsup_{\ep\to 0}\Big|\frac{1}{\ep^{2\alpha-2}}\intom{W_{el}(\nep y^{\ep}(P^{\ep})^{-1})}-\intom{Q(M_{\alpha})}-\frac{o(\ep^{2\alpha-2})}{\ep^{2\alpha-2}}\Big|\leq \delta\intom{Q(M_{\alpha})}.$$
Claims \eqref{clrs1} and \eqref{clrs2} follow now by letting $\delta$ tend to zero.
\end{proof}
\section{Convergence of minimizers and characterization of the limit functional}
\label{finsec}
In this section we deduce convergence of almost minimizers of the three-dimensional energies to minimizers of the limit functional and we show some examples where a characterization of the limit functional can be provided in terms of two-dimensional quantities.

The compactness and liminf inequalities proved in Theorem \ref{liminfineq} and the limsup inequality deduced in Theorem \ref{limsupineq} allow us to obtain the main result of the paper:
\begin{teo}
\label{minpts}
Assume that $\alpha\geq 3$ { and $\gamma_d$ is a finite union of disjoint (nontrivial) closed intervals in the relative topology of $\partial \omega$}. Let $p^0\in L^{\infty}(\Omega;\md)$ be such that there exists a sequence $(p^{\ep,0})\subset L^{\infty}(\Omega;\md)$ satisfying
\begin{eqnarray*}
&&\|p^{\ep,0}\|_{L^{\infty}(\Omega;\md)}\leq C,\\
&&p^{\ep,0}\to p^0\quad\text{strongly in }L^1(\Omega;\md).
\end{eqnarray*}
Assume also that for every $\ep$ the map $P^{\ep,0}:=Id+\ep^{\alpha-1}p^{\ep,0}$ satisfies $\det P^{\ep,0}=1$ a.e. in $\Omega$. Let $\pep$ be defined as in \eqref{defbddat} and let $\cal{J}^{\ep}_{\alpha}$ and $\cal{J}_{\alpha}$ be the functionals given by \eqref{scaleden}, \eqref{numlimfa} and \eqref{numlimf3}. For every $\ep>0$, let $(y^{\ep},P^{\ep})\in\cal{A}_{\ep}(\pep)$ be such that
\begin{equation}
\label{defam}
\cal{J}^{\ep}_{\alpha}(y^{\ep},P^{\ep})-\inf_{(y,P)\in\cal{A}_{\ep}(\pep)}\cal{J}^{\ep}_{\alpha}(y,P)\leq s_{\ep},
\end{equation}
where $s_{\ep}\to 0^+$ as $\ep\to 0$. Finally, let $u^{\ep}$, $v^{\ep}$ and $p^{\ep}$ be the displacements and scaled plastic strain introduced in \eqref{inplane}, \eqref{outofplane} and \eqref{defpep}.
Then, there exists a triple $(u,v,p)\in \cal{A}(u^0,v^0)$ such that, up to subsequences, there holds
\begin{eqnarray}
\label{cmu}&&u^{\ep}\to u\quad\text{strongly in }W^{1,2}(\omega;\R^2),\\
\label{cmv} &&v^{\ep}\to v\quad\text{strongly in }W^{1,2}(\omega),\\
\label{cmp}&&p^{\ep}\to p\quad\text{strongly in }L^2(\Omega;\mthree).
\end{eqnarray}
Moreover, $(u,v,p)$ is a minimizer of $\cal{J}_{\alpha}$ and
\begin{equation}
\label{cme}
\lim_{\ep\to 0}\cal{J}^{\ep}_{\alpha}(y^{\ep},P^{\ep})=\cal{J}_{\alpha}(u,v,p).
\end{equation}
\end{teo}
\begin{proof}[Proof]
By Theorems \ref{liminfineq} and \ref{limsupineq} and by standard arguments in $\Gamma$-convergence we deduce \eqref{cmv}, we show that
\begin{eqnarray*}
&&u^{\ep}\deb u\quad\text{weakly in }W^{1,2}(\omega;\R^2),\\
&&p^{\ep}\deb p\quad\text{weakly in }L^2(\Omega;\mthree),
\end{eqnarray*}
where $(u,v,p)\in \cal{A}(u^0,v^0)$ is a minimizer of $\cal{J}_{\alpha}$, and we prove \eqref{cme}. Strong convergence of $u^{\ep}$ and $p^{\ep}$ follows by \eqref{cme} and by adaptating \cite[Corollaries 5.2 and 5.3]{D2}.
\end{proof}

We remark that the limit plastic strain $p$ depends nontrivially on the $x_3$ variable. Therefore, the limit functionals $\cal{J}_{\alpha}$ cannot, in general, be expressed in terms of two-dimensional quantities only. A characterization of the functionals in terms of the zeroth and first order moments of $p$ can be obtained arguing as follows. Denote by $\pzero, \puno \in L^2(\omega;\md)$ and $\pdue\in L^2(\Omega;\md)$ the following orthogonal components (in the sense of $L^2(\Omega;\md)$) of the plastic strain $p$:
$$\pzero(x'):=\intt{p(x',x_3)},\quad\puno(x'):=12\intt{x_3 p(x',x_3)}\quad\text{for a.e. }x'\in\omega,$$
and
$$\pdue(x):=p(x)-\pzero(x')-x_3\puno(x')\quad\text{for a.e. }x\in\Omega.$$
Then the functionals $\cal{J}_{\alpha}$ can be written in terms of $\pzero,\puno, \pdue$ as
\begin{eqnarray*}
\cal{J}_{\alpha}(u,v,p)&&=\intomm{Q_2(\sym\nabla' u-\pzero')}+\frac{1}{12}\intomm{Q_2((\nabla')^2 v+\puno')}\\
&&+\intom{Q_2(p_{\perp}')}+\intomm{\B{\pzero}}+\frac{1}{12}\intomm{\B{\puno}}\\
&&+\intom{\B{\pdue}}+\intom{H_D(p-p^0)},
\end{eqnarray*}
for $\alpha>3$, and
\begin{eqnarray*}
\cal{J}_3(u,v,p)&&=\intomm{Q_2\big(\sym\nabla' u+\tfrac{1}{2}\nabla' v\otimes\nabla' v-\pzero'\big)}\\
&&+\frac{1}{12}\intomm{Q_2((\nabla')^2 v+\puno')}+\intom{Q_2(\pdue')}+\intomm{\B{\pzero}}\\
&&+\frac{1}{12}\intomm{\B{\puno}}+\intom{\B{\pdue}}+\intom{H_D(p-p^0)},
\end{eqnarray*}
for every $(u,v,p)\in \cal{A}(u^0,v^0)$.

Under additional hypothesis on the boundary data and the preexistent limit plastic strain $p^0$, some two-dimensional characterizations of the limit model can be deduced in the case $\alpha>3$. To this purpose, we introduce the reduced functionals
\begin{equation}
\label{redmom0}
\bar{\cal{J}}_{\alpha}(u,\pzero):=\intomm{Q_2(\sym\nabla' u-\pzero')}+\intomm{\B{\pzero}}+\intomm{H_D(\pzero-\pzero^0)}
\end{equation}
for every $(u,\pzero)\in W^{1,2}(\omega;\R^2)\times L^2(\omega;\md)$ such that $u=u^0$ $\cal{H}^1$ - a.e. on $\gamma_d$, and
\begin{equation}
\label{redmom1}
\hat{\cal{J}}_{\alpha}(v,\puno):=\intomm{Q_2((\nabla')^2 v+\puno')}\\
\intomm{\B{\puno}}+\intomm{H_D(\puno-\puno^0)},
\end{equation}
for every $(v,\puno)\in W^{2,2}(\omega)\times L^2(\omega;\md)$ such that $v=v^0$ and $\nabla' v=\nabla' v^0$ $\cal{H}^1$ - a.e. on $\gamma_d$.

We first show an example where $\cal{J}_{\alpha}$ reduces to $\bar{\cal{J}}_{\alpha}$, that is the limit model depends just on the in-plane displacement and the zeroth moment of the plastic strain.
\begin{teo}
\label{teoz}
Under the hypothesis of Theorem \ref{minpts}, if $\alpha> 3$, $p^0=\bar{p}^0$, with $\bar{p}^0\in L^{\infty}(\omega;\md)$, and $v^0=0$
then, denoting by $\pzero$ the zeroth moment of the limit plastic strain $p$, the pair $(u,\bar{p})$ is a minimizer of $\bar{\cal{J}}_{\alpha}$ and
$$\lim_{\ep\to 0}\cal{J}^{\ep}_{\alpha}(y^{\ep},P^{\ep})=\bar{\cal{J}}_{\alpha}(u,\bar{p}).$$
\end{teo}
\begin{proof}
By Jensen inequality, 
$$\intom{H_D(p-p^0)}\geq \intomm{H_D(\pzero-\pzero^0)},$$
{{hence there holds
$$\cal{J}_{\alpha}(u,v,p)\geq \bar{\cal{J}}_{\alpha}(u,\bar{p}).$$
On the other hand, by setting
$$\tilde{P}^{\ep}:=\exp{(\ep^{\alpha-1}\bar{p})}$$
and 
$$\tilde{y}^{\ep}:=\Big(\begin{array}{c}x',\\\ep x_3\end{array}\Big)+\ep^{\alpha-1}\Big(\begin{array}{c}u\\0\end{array}\Big)+\ep^{\alpha}\int_{-\tfrac 12}^{x_3}{d(x',s)\,ds},$$
with $d\in C^{\infty}_c(\Omega;\R^3)$, then $(\tilde{y}^{\ep}, \tilde{P}^{\ep})\in\cal{A}(\pep) $ and an adaptation of Theorem \ref{limsupineq} yields
$$\lim_{\ep\to 0}\cal{J}^{\ep}_{\alpha}(\tilde{y}^{\ep},\tilde{P}^{\ep})=\bar{\cal{J}}_{\alpha}(u,\bar{p}).$$
By combining the previous remarks we have
\begin{eqnarray*}
\cal{J}_{\alpha}(u,v,p)\geq\bar{\cal{J}}_{\alpha}(u,\bar{p})=\lim_{\ep\to 0}\cal{J}^{\ep}_{\alpha}(\tilde{y}^{\ep},\tilde{P}^{\ep})\geq\lim_{\ep\to 0}\cal{J}^{\ep}_{\alpha}({y}^{\ep},{P}^{\ep}).
\end{eqnarray*}
The conclusion follows now by Theorem \ref{minpts}.}}
\end{proof}
We conclude this section by providing an example where, if $H_D$ is homogeneous of degree one, the $\Gamma$-limit $\cal{J}_{\alpha}$ reduces to $\hat{\cal{J}}_{\alpha}$, that is the limit model depends just on the out-of-plane displacement and the first order moment of the plastic strain. 
\begin{teo}
Assume the function $H_D$ to be homogeneous of degree one, i.e.,
\begin{equation}
\label{1hom}
H_D(\lambda \xi) = |\lambda|H_D(\xi)\quad\text{for every }\lambda\in\R,\, \xi\in\mthree.
\end{equation} 
Under the hypothesis of Theorem \ref{minpts}, if $\alpha> 3$, $p^0=x_3\hat{p}^0$, with $\hat{p}^0\in L^{\infty}(\omega;\md)$, and $u^0=0$
then, denoting by $\puno$ the first order moment of the limit plastic strain $p$, the pair $(v,\puno)$ is a minimizer of $\hat{\cal{J}}_{\alpha}$ and
$$\lim_{\ep\to 0}\cal{J}^{\ep}_{\alpha}(y^{\ep},P^{\ep})=\frac{1}{12}\hat{\cal{J}}_{\alpha}(v,\puno).$$
\end{teo}
\begin{proof}
By Jensen inequality and \eqref{1hom} we deduce, 
$$\intom{H_D(p-p^0)}\geq \intom{|x_3|H_D(p-p^0)}=\intom{H_D(x_3p-x_3p^0)}\\
\geq \tfrac{1}{12}\intomm{H_D(\puno-\puno^0)},$$
which in turn implies {{
$$\cal{J}_{\alpha}(u,v,p)\geq \frac{1}{12}\hat{\cal{J}}_{\alpha}(v,\puno).$$
On the other hand, by setting
$$\tilde{P}^{\ep}:=\exp{(\ep^{\alpha-1}x_3\hat{p})}$$
and 
$$\tilde{y}^{\ep}:=\Big(\begin{array}{c}x',\\\ep x_3\end{array}\Big)-\ep^{\alpha-1}x_3\Big(\begin{array}{c}\nabla' v\\0\end{array}\Big)+\ep^{\alpha-2}\Big(\begin{array}{c}v\\0\end{array}\Big)+\ep^{\alpha}\int_{-\tfrac 12}^{x_3}{d(x',s)\,ds},$$
with $d\in C^{\infty}_c(\Omega;\R^3)$, an adaptation of Theorem \ref{limsupineq} yields
$$\lim_{\ep\to 0}\cal{J}^{\ep}_{\alpha}(\tilde{y}^{\ep},\tilde{P}^{\ep})=\frac{1}{12}\hat{\cal{J}}_{\alpha}(v,\hat{p}).$$
The conclusion follows now arguing as in the proof of Theorem \ref{teoz}.}}
\end{proof}

\noindent
\textbf{Acknowledgements.}
I warmly thank Maria Giovanna Mora for having proposed to me the study of this problem and for many helpful and stimulating discussions and suggestions.\\
This work was partially supported by MIUR under PRIN 2008.

\end{document}